\documentclass[12pt,a4paper]{article}
\usepackage{amsmath}
\usepackage{amssymb}
\usepackage{amsthm}

\newtheorem{defi}{Definition}[section]
\newtheorem{thm}[defi]{Theorem}

\newtheorem{lemma}[defi]{Lemma}
\newtheorem{cor}[defi]{Corollary}

\theoremstyle{remark}
\newtheorem{rem}[defi]{Remark}

\DeclareFontEncoding{OT2}{}{}
\DeclareFontSubstitution{OT2}{cmr}{m}{n}
\DeclareFontFamily{OT2}{cmr}{\hyphenchar\font45 }
\DeclareFontShape{OT2}{cmr}{m}{n}{%
	<5><6><7><8><9>gen*wncyr%
	<10><10.95><12><14.4><17.28><20.74><24.88>wncyr10}{}
\DeclareFontShape{OT2}{cmr}{b}{n}{%
	<5><6><7><8><9>gen*wncyb%
	<10><10.95><12><14.4><17.28><20.74><24.88>wncyb10}{}
\DeclareMathAlphabet{\mathcyr}{OT2}{cmr}{m}{n}
\DeclareMathAlphabet{\mathcyb}{OT2}{cmr}{b}{n}
\SetMathAlphabet{\mathcyr}{bold}{OT2}{cmr}{b}{n}

\newcommand{\qbinom}[2]{\genfrac{[}{]}{0pt}{}{#1}{#2}} 

\title{A $q$-analogue of Boyadzhiev-Mneimneh-type binomial sums of
finite multi-polylogarithms}

\date{\empty}
\author{Ken Kamano}

\begin{document}

\maketitle

\begin{abstract}
We give a formula for a $q$-analogue of 
Boyadzhiev-Mneimneh-type binomial sums of 
finite multi-polylogarithms.
In the limit as $q \to 1$, this formula reduces to an identity
equivalent to the Sakugawa-Seki identities.
We also give a formula for Boyadzhiev-Mneimneh-type sums
corresponding to the Cauchy binomial theorem.
\end{abstract}

\section{Introduction and main results}

Let $n$ be a positive integer and $p$ a real number with $0\le p\le 1$.
Mneimneh \cite{M} proved the following simple but non-trivial identity:
\begin{equation}\label{eq:Mneimneh}
\sum_{k=1}^n H_k \binom{n}{k} p^k(1-p)^{n-k}
= \sum_{i=1}^n \frac{1-(1-p)^i}{i},
\end{equation}
where $H_k:=\sum_{i=1}^k 1/i$ is the $k$-th harmonic number.
The left-hand side of \eqref{eq:Mneimneh}
is the mean of $H_k$ under 
the binomial random variable $S_n$,
and it has some applications to probabilistic problems
(see \cite{M}).
It is worth noting that
Boyadzhiev \cite[Prop.~6]{Boya} proved
the following identity, which is a generalization of 
\eqref{eq:Mneimneh}:
\begin{align}\label{eq:Boya}
\sum_{k=1}^n\binom{n}{k}H_k \lambda^{n-k} \mu^k
= H_n(\lambda+\mu)^{n} - 
\sum_{j=1}^{n}\frac{\lambda^j(\lambda+\mu)^{n-j}}{j}
\ \ (n\ge 1,\ \lambda,\mu\in\mathbb{C}).
\end{align}

Recently, Mneimneh's identity \eqref{eq:Mneimneh}
and its generalization \eqref{eq:Boya}
have been reproved and 
extended to hyperharmonic numbers, multiple harmonic sums
and finite analogues of polylogarithms
(see e.g., \cite{C}, \cite{KW},
\cite{PX} and \cite{TZ}).
For $\boldsymbol{s}=(s_1,\ldots, s_d) \in \mathbb{Z}_{>0}^d$ and $a\in\mathbb{R}$, Gen\v{c}ev \cite{G} introduced the following Boyadzhiev-Mneimneh-type sums:
\begin{align*}
M_n^{(\boldsymbol{s})}(a,p):=
\sum_{k=1}^n \binom{n}{k} p^k(1-p)^{n-k}
\zeta_k^{\star}(s_1,\ldots, s_d;a),
\end{align*}
where 
\[ \zeta_k^{\star}(s_1,\ldots, s_d;a)
:= \sum_{k\ge n_1\ge \cdots \ge n_d\ge 1}
\frac{a^{n_d}}{n_1^{s_1}\cdots n_d^{s_d}}.
\]
When $s_1\ge 2$ and $a=1$, the value
\[ \zeta^{\star}(s_1,\ldots, s_d):=\lim_{k\to\infty} 
\zeta_k^{\star}(s_1,\ldots, s_d;1)
=\sum_{n_1\ge \cdots \ge n_d\ge 1}
\frac{1}{n_1^{s_1}\cdots n_d^{s_d}}\]
is convergent and
known as a \textit{multiple zeta star value of 
index} $\boldsymbol{s}=(s_1,\ldots ,s_d)$.

For an index $\boldsymbol{s}=(s_1,\ldots ,s_d) \in \mathbb{Z}_{>0}^d$, 
we call $\text{dep } \boldsymbol{s}:=d$ its depth.
We also use the notation
$l(0):=0$, $l(i):=s_1+\cdots +s_i$\ ($1\le i \le d$) and $w:=l(d)=s_1+\cdots +s_d$.   
As a generalization of \eqref{eq:Mneimneh},
Gen\v{c}ev \cite{G} proved the following formula,
which generalizes the one conjectured by Pan and Xu \cite[Sec.~4]{PX}.

\begin{thm}[{\cite[Theorem 2]{G}}]\label{thm:Gencev's formula}
For any index $\boldsymbol{s}=(s_1,\ldots ,s_d) \in \mathbb{Z}_{>0}^d$, $n\in\mathbb{Z}_{>0}$ and real numbers $a$ and $p$,
the following identity holds:
\begin{align*}
M_n^{(\boldsymbol{s})}(a,p)=
\sum_{n\ge n_1\ge \cdots \ge n_w\ge 1}
\dfrac{ \prod_{r=1}^d (1-p)^{n_{l(r-1)+1}-n_{l(r)}}  }
      { n_1\cdots n_w}
\left(  (1-p+ap)^{n_w}-(1-p)^{n_w} \right).
\end{align*}
\end{thm}

Let $q$ be a complex number with $|q|\neq 0, 1$.
For a non-negative integer $n$, set a $q$-integer 
$[n]:=\frac{1-q^n}{1-q}=1+q+\cdots +q^{n-1}$.
The $q$-factorial $[n]!$ is defined by
$[n]!:=[n][n-1]\cdots [1]$ ($n\ge 1$) and $[0]!:=1$.
The $q$-binomial coefficients $\qbinom{n}{k}$ are defined by
$\qbinom{n}{k}:=\frac{[n]!}{[k]![n-k]!}$
for $n\ge k \ge 0$. We set $\qbinom{n}{k}=0$ if $k>n$,
as usual. In the limit as $q \to 1$, the values
$[n]$, $[n]!$ and $\qbinom{n}{k}$ reduce to the ordinary
$n$, $n!$ and $\binom{n}{k}$, respectively.

Let $R$ be a polynomial ring $\mathbb{C}[t_1,\ldots ,t_d]$ in $d$ variables over $\mathbb{C}$.
We denote by
$\mathbb{C}_q[x,y]$ (resp.~$R_q[x,y]$) the associative
unital algebra over $\mathbb{C}$ (resp.~$R$) generated
by $x$ and $y$ with a relation $yx=qxy$.
For example, it holds that $yx^2 = qxyx = q^2x^2y$ 
in $\mathbb{C}_q[x,y]$ or $R_q[x,y]$.

The following theorem, attributed to 
Sch\"{u}tzenberger \cite{S}, is a well-known analogue
of the classical binomial theorem (see also e.g.,~\cite{K}).
\begin{thm}
The following identity holds in $\mathbb{C}_q[x,y]$:
\begin{align}\label{eq:q-binomial_nc}
(x+y)^n=\sum_{k=0}^n \qbinom{n}{k}x^ky^{n-k}\ \ (n\ge 0).
\end{align}
\end{thm}

For $k\in\mathbb{Z}_{>0}$, $\boldsymbol{s}=(s_1,\ldots ,s_d)\in \mathbb{Z}_{>0}^d$ and 
$\boldsymbol{t}=(t_1,\ldots ,t_d)$, 
define a $q$-analogue of finite multi-polylogarithms
as
\begin{align*}
 l_k^{\star,q}(\boldsymbol{s},\boldsymbol{t})
 :=\sum_{k\ge n_1\ge \cdots \ge n_d\ge 1}
   \dfrac{t_1^{n_1-n_2} \cdots  t_{d-1}^{n_{d-1}-n_d} t_d^{n_d}}  { [n_1]^{s_1} \cdots [n_d]^{s_d}}\ \in R.
\end{align*}
We set $l_k^{\star}(\boldsymbol{s},\boldsymbol{t}):= \lim_{q\to 1} l_k^{\star,q}(\boldsymbol{s},\boldsymbol{t})$, i.e.,
\begin{align*}
l_k^{\star}(\boldsymbol{s},\boldsymbol{t})
=\sum_{k\ge n_1\ge \cdots \ge n_d\ge 1}
   \dfrac{t_1^{n_1-n_2} \cdots  t_{d-1}^{n_{d-1}-n_d} t_d^{n_d}}  { n_1^{s_1} \cdots n_d^{s_d}}.
\end{align*}
The function $l_k^{\star}(\boldsymbol{s},\boldsymbol{t})
$ is a finite analogue of
the so-called multi-polylogarithms of shuffle type.
We note that
$l_k^{\star}(\boldsymbol{s},(1,\ldots , 1, a))
= \zeta_k^{\star}(\boldsymbol{s},a)$.
We define Boyadzhiev-Mneimneh-type binomial sums of $l_k^{\star,q}(\boldsymbol{s},\boldsymbol{t})$ as
\begin{align*} 
M_n^{q}(\boldsymbol{s}, \boldsymbol{t}; x,y)
:= \sum_{k=1}^n \qbinom{n}{k}x^ky^{n-k}
   l_k^{\star,q}(\boldsymbol{s},\boldsymbol{t})\ \ (n\ge 0)
\end{align*}
as an element of $R_q[x,y]$.
In the limit as $q \to 1$, the variables $x$ and $y$ commute,
so we can consider
$M_n(\boldsymbol{s}, \boldsymbol{t}; x,y):=
\lim_{q\to 1} M_n^{q}(\boldsymbol{s}, \boldsymbol{t}; x,y)$ as an element of the ordinary polynomial ring $R[x,y]$.

The aim of the present paper is 
to give an expression of Boyadzhiev-Mneimneh-type binomial sums
$M_n^{q}(\boldsymbol{s}, \boldsymbol{t}; x,y)$,
which generalizes Theorem \ref{thm:Gencev's formula}.
The following is the main theorem of the present paper.

\begin{thm}\label{thm:mainthm}
For $n\ge 1$, the following identity holds in $R_q[x,y]$:
\begin{equation}\label{eq:mainthm}
\begin{split}
M_n^{q}(\boldsymbol{s}, \boldsymbol{t}; x,y)
&=\sum_{n\ge n_1\ge \cdots \ge n_w\ge 1}
 \frac{(x+y)^{n-n_1}}{ [n_1]\cdots [n_w]}
  \left( \prod_{r=1}^{d-1} y^{n_{l(r-1)+1}-n_
{l(r)}}(t_rx+y)^{n_{l(r)}-n_{l(r)+1}} \right)\\[10pt]
&\hspace{20pt}  \times y^{n_{l(d-1)+1}-n_{l(d)}}
  \left( (t_dx+y)^{n_{l(d)}} - y^{n_{l(d)}}   \right),
\end{split}
\end{equation}
where the order in the product symbol is defined as
$\prod_{j=1}^{v}X_j=X_1X_2\cdots X_v$.
\end{thm}

Taking the limit as $q\to 1$ in Eq.~\eqref{eq:mainthm}
yields the following corollary. 
Applying $\boldsymbol{t} =(1,\ldots , 1,a)$, $x=p$ and $y=1-p$
in this corollary, 
we can obtain Theorem \ref{thm:Gencev's formula}.

\begin{cor}\label{cor:maincor}
For $n\ge 1$, the following identity holds in 
$R[x,y]$:
\begin{equation}\label{eq:maincor}
\begin{split}
    M_n(\boldsymbol{s}, \boldsymbol{t}; x,y)
&=\sum_{n\ge n_1\ge \cdots \ge n_w\ge 1}
 \frac{(x+y)^{n-n_1} y^
{n_1+n_{l(1)+1}+\cdots +n_{l(d-1)+1}-n_{l(1)}-\cdots -n_{l(d)}}}{n_1\cdots n_w} \\[10pt]
&\hspace{20pt}  \times \left( \prod_{r=1}^{d-1} (t_rx+y)^{n_{l(r)}-n_{l(r)+1}}\right)
  \left((t_dx+y)^{n_{l(d)}} - y^{n_{l(d)}}\right).
\end{split}
\end{equation}
\end{cor}

The present paper is organized as follows.
In Section \ref{sec:Pf_mainhhm}
we prove our main theorem (Theorem \ref{thm:mainthm}).
In Section \ref{sec:Sakugawa-Seki}
we discuss the connection between our main results and 
identities given by Sakugawa and Seki \cite{SS}.
More precisely, we show that Eq.~\eqref{eq:maincor} in Corollary
\ref{cor:maincor} and 
the Sakugawa-Seki identities are equivalent.
In Section \ref{sec:Cauchybt}
we give a formula for Boyadzhiev-Mneimneh-type Cauchy binomial sums of
$l_n^{\star, q}$.

\section{Proof of Main Theorem}\label{sec:Pf_mainhhm}

It is well-known that $q$-binomial coefficients satisfy the following identity:
\begin{align}\label{eq:q-binomial_rec}
   \qbinom{n}{k} = q^k\qbinom{n-1}{k} + \qbinom{n-1}{k-1}
\ \ \ (n, k \ge 1).
\end{align}
By using this identity, we obtain the following lemma.

\begin{lemma}\label{lemma:keylemma_a}
For positive integers $n$ and $k$ with $n\ge k\ge 1$, we have
\begin{align}\label{eq:keylemma_a}
         \qbinom{n}{k}\dfrac{1}{[k]^s}
    =\displaystyle\sum_{n\ge n_1\ge \cdots \ge n_{s}\ge k}
        \dfrac{q^{(n-n_{s})k}}
              {[n_1]\cdots [n_{s}]} 
        \qbinom{n_{s}}{k}\ \ \ (s\ge 1).
\end{align}
\end{lemma}

\begin{proof}
We prove the identity by induction on $s$.
By using Eq.~\eqref{eq:q-binomial_rec} repeatedly,
we have
\[  \qbinom{n}{k}=\sum_{n_1=k}^n q^{(n-n_1)k}\qbinom{n_1-1}{k-1}.   \]
Hence we have
\begin{equation}\label{eq:keylemma_s=1}
    \qbinom{n}{k}\frac{1}{[k]}
= \sum_{n_1=k}^n \frac{q^{(n-n_1)k}}{[n_1]}\qbinom{n_1}{k}
\end{equation}  
and Eq.~\eqref{eq:keylemma_a} holds for $s=1$.

Assume that Eq.~\eqref{eq:keylemma_a} holds for some $s\ge 1$.
Then by inductive assumption, we have
\begin{align*}
         \qbinom{n}{k}\dfrac{1}{[k]^{s+1}}
=\qbinom{n}{k}\dfrac{1}{[k]^{s}} \dfrac{1}{[k]}
=        \displaystyle\sum_{n\ge n_1\ge \cdots \ge n_s\ge k}
        \dfrac{q^{(n-n_s)k}}
              {[n_1]\cdots [n_{s}]} 
        \qbinom{n_{s}}{k}  \frac{1}{[k]}.
 \end{align*}
By using Eq.~\eqref{eq:keylemma_s=1} again,
this equals
\begin{align*} \displaystyle\sum_{n\ge n_1\ge \cdots \ge n_{s+1}\ge k}
        \dfrac{q^{(n-n_{s+1})k}}
              {[n_1]\cdots [n_{s+1}]} 
        \qbinom{n_{s+1}}{k}.
\end{align*}
This shows that Eq.~\eqref{eq:keylemma_a} holds for $s+1$
and this completes the proof.
\end{proof}

The following lemma is a kind of
recurrence relations for
$M_n^{q}(\boldsymbol{s}, \boldsymbol{t}; x,y)$.

\begin{lemma}\label{lemma:induction_lemma}
For any integer $n\ge 1$ and 
index $(s_1,\ldots ,s_d)\in\mathbb{Z}_{>0}^d$,
the following identities hold:
\begin{equation}\label{eq:rec_M_n^q}
\begin{split}
&M_n^{q}(\boldsymbol{s}, \boldsymbol{t}; x,y)
-(x+y)M_{n-1}^{q}(\boldsymbol{s}, \boldsymbol{t}; x,y)\\
& = \begin{cases}
  \dfrac{1}{[n]}\displaystyle \sum_{n\ge n_1\ge \cdots \ge n_{s_1-1}\ge 1} 
        \dfrac{y^{n-n_{s_1-1}}}{[n_1]\cdots [n_{s_1-1}]}
        M_{n_{s_1-1}}((s_2,\ldots ,s_d), (\tfrac{t_2}{t_1},\ldots , \tfrac{t_d}{t_1} ); t_1x,y)& (d\ge 2, s_1\ge 2),\\
  \dfrac{1}{[n]} M_{n}^q((s_2,\ldots ,s_d), (\tfrac{t_2}{t_1},\ldots ,        \tfrac{t_d}{t_1} ); t_1x,y)& (d\ge 2, s_1=1),\\
 \dfrac{1}{[n]}\displaystyle \sum_{n\ge n_1\ge \cdots \ge n_{s_1-1}\ge 1} 
        \dfrac{y^{n-n_{s_1-1}}}{[n_1]\cdots [n_{s_1-1}]}
        \left( (t_1x+y)^{n_{s_1-1}}-y^{n_{s_1-1}} \right)& (d=1, s_1\ge 2),\\
 \dfrac{1}{[n]} \left( (t_1x+y)^{n}-y^{n} \right)& (d=1, s_1=1).
 \end{cases}
\end{split}
\end{equation}
\end{lemma}

\begin{rem}
When both $\boldsymbol{s}$ and $\boldsymbol{t}$ are the empty index $\emptyset$,
we may define $l_k^{\star,q}(\emptyset, \emptyset)=1$ for 
any $k\ge 1$.
Then 
$M_n^q(\emptyset, \emptyset; x,y)
= \sum_{k=1}^n \qbinom{n}{k} x^k y^{n-k} \cdot 1
= (x+y)^n-y^n$ and
Eq.~\eqref{eq:rec_M_n^q} above can be written simply as 
\begin{equation*}
\begin{split}
&M_n^{q}(\boldsymbol{s}, \boldsymbol{t}; x,y)
-(x+y)M_{n-1}^{q}(\boldsymbol{s}, \boldsymbol{t}; x,y)\\
& = \begin{cases}
  \dfrac{1}{[n]}\displaystyle \sum_{n\ge n_1\ge \cdots \ge n_{s_1-1}\ge 1} 
        \dfrac{y^{n-n_{s_1-1}}}{[n_1]\cdots [n_{s_1-1}]}
        M_{n_{s_1-1}}^q((s_2,\ldots ,s_d), (\tfrac{t_2}{t_1},\ldots , \tfrac{t_d}{t_1} ); t_1x,y)& (s_1\ge 2),\\
  \dfrac{1}{[n]} M_{n}^q((s_2,\ldots ,s_d), (\tfrac{t_2}{t_1},\ldots ,        \tfrac{t_d}{t_1} ); t_1x,y)& (s_1=1).
  \end{cases}
\end{split}
\end{equation*}
\end{rem}

\begin{proof}[Proof of Lemma \ref{lemma:induction_lemma}]
First, let us consider the case $d\ge 2$.
By Eq.~\eqref{eq:q-binomial_rec}, we have
\begin{align*}
M_n^{q}(\boldsymbol{s}, \boldsymbol{t}; x,y)
&=\sum_{k=1}^n\left( q^k\qbinom{n-1}{k}+\qbinom{n-1}{k-1} \right)
x^ky^{n-k} l_k^{\star,q}(\boldsymbol{s}, \boldsymbol{t})\\
&=\sum_{k=1}^nq^k\qbinom{n-1}{k}
x^ky^{n-k} l_k^{\star,q}(\boldsymbol{s}, \boldsymbol{t})
 + \sum_{k=1}^n\qbinom{n-1}{k-1}
x^ky^{n-k} l_k^{\star,q}(\boldsymbol{s}, \boldsymbol{t})  \\
&= y\sum_{k=1}^{n-1}\qbinom{n-1}{k}x^{k}y^{n-1-k} 
    l_k^{\star,q}(\boldsymbol{s}, \boldsymbol{t}) 
  +x \sum_{k=2}^n\qbinom{n-1}{k-1}x^{k-1}y^{n-1-(k-1)} 
    l_{k-1}^{\star,q}(\boldsymbol{s}, \boldsymbol{t}) \\
& \hspace{10pt} 
  +\sum_{k=1}^n\qbinom{n-1}{k-1}x^{k}y^{n-k} 
       \frac{1}{[k]^{s_1}} 
       \sum_{k\ge n_2 \ge \cdots \ge n_d \ge 1}
          \frac{ t_1^{k-n_2} t_2^{n_2-n_3} \cdots t_{d-1}^{n_{d-1}-n_d} t_d^{n_d} } 
               {[n_2]^{s_2} \cdots [n_d]^{s_d}}\\
&=(x+y) \sum_{k=1}^{n-1}\qbinom{n-1}{k}x^{k}y^{n-1-k} 
    l_k^{\star,q}(\boldsymbol{s}, \boldsymbol{t}) \\
&\hspace{10pt} + \sum_{k=1}^n\qbinom{n-1}{k-1}\frac{1}{[k]^{s_1}}
(t_1x)^{k}y^{n-k}\sum_{k\ge n_2 \ge \cdots \ge n_d \ge 1}
          \frac{ (\tfrac{t_2}{t_1})^{n_2-n_3} \cdots (\tfrac{t_{d-1}}{t_1})^{n_{d-1}-n_d} 
          (\tfrac{t_d}{t_1})^{n_d}  } 
               {[n_2]^{s_2} \cdots [n_d]^{s_d}}.
\end{align*}
Hence we obtain that
\begin{equation}\label{eq:M_n-M_n-1}
\begin{split}
    &M_n^{q}(\boldsymbol{s}, \boldsymbol{t}; x,y)
-(x+y)M_{n-1}^{q}(\boldsymbol{s}, \boldsymbol{t}; x,y)\\
&= \frac{1}{[n]}
 \sum_{k=1}^n\qbinom{n}{k}\frac{1}{[k]^{s_1-1}}
 (t_1x)^{k}y^{n-k}        
        l_k^{\star,q} \left( (s_2,\ldots ,s_d),  (\tfrac{t_2}{t_1},  \ldots,  \tfrac{t_d}{t_1}) \right).
\end{split}
\end{equation}

When $s_1\ge 2$, 
by Lemma \ref{lemma:keylemma_a}, 
the right-hand side of \eqref{eq:M_n-M_n-1} equals
\begin{align*}
&\frac{1}{[n]} \sum_{k=1}^n\sum_{n\ge n_1\ge \cdots n_{s_1-1}
\ge k}
\frac{q^{(n-n_{s_1-1})k}}{[n_1]\cdots [n_{s_1-1}]}
\qbinom{n_{s_1-1}}{k} (t_1x)^k y^{n-k} 
l_k^{\star,q}
       \left( (s_2,\ldots ,s_d),  (\tfrac{t_2}{t_1},  \ldots,  \tfrac{t_d}{t_1}) \right)\\
&=\frac{1}{[n]} \sum_{k=1}^n \sum_{n\ge n_1\ge \cdots n_{s_1-1}
\ge k}
\frac{y^{n-n_{s_1-1}}}{[n_1]\cdots [n_{s_1-1}]} 
\qbinom{n_{s_1-1}}{k} (t_1x)^k y^{n_{s_1-1}-k} 
l_k^{\star,q}
       \left( (s_2,\ldots ,s_d),  (\tfrac{t_2}{t_1},  \ldots,  \tfrac{t_d}{t_1}) \right)\\
&=\frac{1}{[n]} \sum_{n\ge n_1\ge \cdots n_{s_1-1}
\ge 1}
\frac{y^{n-n_{s_1-1}}}{[n_1]\cdots [n_{s_1-1}]} 
\sum_{k=1}^{n_{s_1-1}} \qbinom{n_{s_1-1}}{k} (t_1x)^k y^{n_{s_1-1}-k} 
l_k^{\star,q}
       \left( (s_2,\ldots ,s_d),  (\tfrac{t_2}{t_1},  \ldots,  \tfrac{t_d}{t_1}) \right)\\
&=\frac{1}{[n]} \sum_{n\ge n_1\ge \cdots n_{s_1-1}
\ge 1}
\frac{y^{n-n_{s_1-1}}}{[n_1]\cdots [n_{s_1-1}]} 
M_{n_{s_1-1}}^q
       \left( (s_2,\ldots ,s_d),  (\tfrac{t_2}{t_1},  \ldots,  \tfrac{t_d}{t_1}); t_1x, y \right).
\end{align*}
When $s_1=1$, 
the right-hand side of \eqref{eq:M_n-M_n-1} equals
\begin{align*}
\frac{1}{[n]} M_n^q \left( (s_2,\ldots ,s_d),  (\tfrac{t_2}{t_1},  \ldots,  \tfrac{t_d}{t_1}) ; t_1x, y \right).
\end{align*}
Therefore Lemma \ref{lemma:induction_lemma} is proved in the case $d\ge 2$.
\vspace{10pt}

Next we consider the case $d=1$.
By the similar calculation in the case $d\ge 2$, we have 
\begin{align} \label{eq:M_n-M_n-1_2}
M_n^q(\boldsymbol{s}, \boldsymbol{t})
-(x+y) M_{n-1}^q(\boldsymbol{s}, \boldsymbol{t})
= \frac{1}{[n]}\sum_{k=1}^n \qbinom{n}{k} \frac{1}{[k]^{s_1-1}}
(t_1x)^k y^{n-k}.
\end{align}

When $s_1\ge 2$, by Lemma \ref{lemma:keylemma_a}
and Eq.~\eqref{eq:q-binomial_nc},
the right-hand side of \eqref{eq:M_n-M_n-1_2}
equals 
\begin{align*}
&\frac{1}{[n]} \sum_{k=1}^n 
\sum_{n\ge n_1 \ge \cdots \ge n_{s_1-1}\ge k}
\frac{q^{(n-n_{s_1-1})k}}{[n_1]\cdots [n_{s_1-1}]}
\qbinom{n_{s_1-1}}{k} (t_1x)^k y^{n-k}\\
&=\frac{1}{[n]} 
\sum_{n\ge n_1 \ge \cdots \ge n_{s_1-1}\ge 1}
\frac{ y^{n-n_{s_1-1}} }{[n_1]\cdots [n_{s_1-1}]}
\sum_{k=1}^{n_{s_1-1}} 
\qbinom{n_{s_1-1}}{k} (t_1x)^k y^{n_{s_1-1}-k}\\
&=
\frac{1}{[n]} 
\sum_{n\ge n_1 \ge \cdots \ge n_{s_1-1}\ge 1}
\dfrac{y^{n-n_{s_1-1}}
((t_1x+y)^{n_{s_1-1}} -y^{n_{s_1-1}} )}
{[n_1]\cdots [n_{s_1-1}]}.
\end{align*}

When $s_1=1$, the right-hand side of \eqref{eq:M_n-M_n-1_2}
equals 
\[ \frac{1}{[n]} ( (t_1x+y)^n - y^n  ).\]
Consequently, Lemma \ref{lemma:induction_lemma} is also proved in the case $d=1$.
\end{proof}
\vspace{10pt}

\begin{proof}[Proof of Theorem \ref{thm:mainthm}]

We show the theorem by the following way:
\begin{enumerate}
\item Prove the theorem in the case $n=1$.
\item Prove the theorem in the case $d$ $(=\text{dep\,}\boldsymbol{s})=1$.
\item Assume that the theorem holds for all 
$M_{n'}(\boldsymbol{s}',\boldsymbol{t}'; x,y)$
with 
\begin{itemize}
\item $n'\le n$ and \text{dep\,}$\boldsymbol{s}'< \text{dep\,}\boldsymbol{s}$,
\item $n'< n$ and \text{dep\,}$\boldsymbol{s}' \le \text{dep\,}\boldsymbol{s}$.
\end{itemize}
Then prove the theorem holds for 
$M_{n}(\boldsymbol{s},\boldsymbol{t}; x,y)$.
\end{enumerate}

\begin{enumerate}

\item The case $n=1$ is clear. In fact, 
both sides of \eqref{eq:mainthm} are $t_dx$ when $n=1$.

\item 
Eq.~\eqref{eq:mainthm} for $d=1$ is 
\begin{align}\label{eq:pfmainthm_d=1}
M_n^q(s,t; x,y)
=\sum_{n\ge n_1\ge \cdots \ge n_s \ge 1}
    \frac{ (x+y)^{n-n_1} y^{n_1-n_s} \{ (tx+y)^{n_s} -y^{n_s}  \} }  {[n_1]\cdots [n_s]}\ \ (n\ge 1).
\end{align}

We prove this equation by induction on $n$.
The case $n=1$ has been proved in (i).

Assume that Eq.~\eqref{eq:pfmainthm_d=1} holds for $n-1$.
For $s\ge 2$, by Lemma \ref{lemma:induction_lemma},
we have
\begin{align*}
& M_n^q(s,t; x,y)\\
&=(x+y)M_{n-1}^q(s,t; x,y)
+\frac{1}{[n]} \sum_{n\ge m_1\ge \cdots \ge m_{s-1}\ge 1}
 \frac{y^{n-m_{s-1}}}{[m_1]\cdots [m_{s-1}]}
   ((tx+y)^{m_{s-1}}-y^{m_{s-1}}).
\end{align*}
By inductive assumption, this equals
\begin{align*}
&(x+y)\sum_{n-1\ge n_1\ge \cdots \ge n_{s-1}\ge 1}
  \frac{(x+y)^{n-1-n_1} y^{n_1-n_s}   ((tx+y)^{n_s}-y^{n_s})}{[n_1]\cdots [n_s]}\\
&\hspace{10pt}+\frac{1}{[n]} \sum_{n\ge m_1\ge \cdots \ge m_{s-1}\ge 1}
 \frac{y^{n-m_{s-1}}}{[m_1]\cdots [m_{s-1}]}
   ((tx+y)^{m_{s-1}}-y^{m_{s-1}})\\
&=\sum_{n\ge n_1\ge \cdots \ge n_{s-1}\ge 1}
  \frac{(x+y)^{n-n_1} y^{n_1-n_s}   ((tx+y)^{n_s}-y^{n_s})}{[n_1]\cdots [n_s]}.
\end{align*}
Hence Eq.~\eqref{eq:pfmainthm_d=1} holds for $n$.
It can be proved similarly in the case $s=1$.

\item 
Let $l'(0):=0$, $l'(i):=s_2+\cdots +s_{i+1}$ ($1\le i \le d-1$) and $w':=l'(d-1)=s_2+\cdots +s_d$.
By Lemma \ref{lemma:induction_lemma} and inductive asummption, we have 
\begin{align*}
&M_n^q(\boldsymbol{s},\boldsymbol{t}; x,y)\\
&=(x+y)M_{n-1}^q(\boldsymbol{s},\boldsymbol{t}; x,y)\\
&+
  \dfrac{1}{[n]}\displaystyle \sum_{n\ge n_1\ge \cdots \ge n_{s_1-1}\ge 1} 
        \dfrac{y^{n-n_{s_1-1}}}{[n_1]\cdots [n_{s_1-1}]}
        M^q_{n_{s_1-1}}((s_2,\ldots ,s_d), (\tfrac{t_2}{t_1},\ldots , \tfrac{t_d}{t_1} ); t_1x,y)\\
&= \sum_{n-1\ge n_1\ge \cdots \ge n_w \ge 1}
 \frac{(x+y)^{n-n_1}}{ [n_1]\cdots [n_w]}
  \left( \prod_{r=1}^{d-1} y^{n_{l(r-1)+1}-n_{l(r)}}(t_rx+y)^{n_{l(r)}-n_{l(r)+1}} \right)\\[10pt]
&\hspace{20pt}  \times y^{n_{l(d-1)+1}-n_{l(d)}}
  \left( (t_dx+y)^{n_{l(d)}} - y^{n_{l(d)}} \right)\\
&+
  \dfrac{1}{[n]}\displaystyle \sum_{n\ge n_1\ge \cdots \ge n_{s_1-1}\ge 1} 
        \dfrac{y^{n-n_{s_1-1}}}{[n_1]\cdots [n_{s_1-1}]}
\sum_{n_{s_1-1} \ge m_1\ge \cdots \ge m_{w'}\ge 1}\\
&\hspace{20pt} \times \frac{(t_1x+y)^{n_{s_1-1}-m_1}}{ [m_1]\cdots [m_{w'}]}
  \left( \prod_{r=1}^{d-2} y^{m_{l'(r-1)+1}-m_{l'(r)}}\left(\tfrac{t_{r+1}}{t_1}t_1x+y\right)^{m_{l'(r)}-m_{l'(r)+1}} \right) \\
&\hspace{20pt} \times y^{m_{l'(d-2)+1}-m_{l'(d-1)}}
  \left( \left(\tfrac{t_{d}}{t_1} t_1x +y\right)^{m_{l'(d-1)}} - y^{m_{l'(d-1)}}   \right).
\end{align*}

The second term coincides with the first term for ``$n_1=n$''.
Hence we obtain that 
\begin{align*}
&M_n^q(\boldsymbol{s},\boldsymbol{t}; x,y)\\
&= \sum_{n\ge n_1\ge \cdots \ge n_w \ge 1}
 \frac{(x+y)^{n-n_1}}{ [n_1]\cdots [n_w]}
  \left( \prod_{r=1}^{d-1} y^{n_{l(r-1)+1}-n_{l(r)}}(t_rx+y)^{n_{l(r)}-n_{l(r)+1}} \right)\\
&\hspace{20pt}  \times y^{n_{l(d-1)+1}-n_{l(d)}}
  \left( (t_dx+y)^{n_{l(d)}} - y^{n_{l(d)}}   \right)
  \end{align*}
and this proves that the theorem holds for $M_n^q(\boldsymbol{s},\boldsymbol{t};x,y)$.
\end{enumerate}
\end{proof}

\section{Sakugawa-Seki identities}\label{sec:Sakugawa-Seki}

In this section, we discuss the connection between
our results and the Sakugawa-Seki identities proved in \cite{SS}.
The followings are the Sakugawa-Seki identities,
which generalize the classical Euler's identity:
\[ \sum_{k=1}^n \binom{n}{k} \frac{(-1)^{k-1}}{k} = \sum_{k=1}^n \frac{1}{k}
\ \ (n\ge 1). \]

\begin{thm}[{\cite[Theorem~2.5]{SS}}] \label{thm:SS}
For $(s_1,\ldots, s_d) \in \mathbb{Z}_{>0}^d$ and
$n\ge 1$, the following identities hold:
\begin{equation}\label{eq:SS_1}
\begin{split}
& \sum_{n\ge n_1\ge \cdots \ge n_d\ge 1}
  (-1)^{n_1}\binom{n}{n_1} 
   \dfrac{t_1^{n_1-n_2} \cdots t_{d-1}^{n_{d-1}-n_d} t_d^{n_d}}
         {n_1^{s_1}\cdots n_d^{s_d}}\\
&=  \sum_{n\ge n_1\ge \cdots \ge n_w\ge 1}
    \dfrac{(1-t_1)^{n_{l(1)}-n_{l(1)+1}} \cdots 
           (1-t_{d-1})^{n_{l(d-1)}-n_{l(d-1)+1}}
            \{ (1-t_d)^{n_{l(d)}} -1 \} }
         {n_1\cdots n_w},
\end{split}
\end{equation}

\begin{equation}\label{eq:SS_2}
\begin{split}
& \sum_{n\ge n_1\ge \cdots \ge n_d\ge 1} 
   \dfrac{t_1^{n_1-n_2} \cdots t_{d-1}^{n_{d-1}-n_d} t_d^{n_d}}
         {n_1^{s_1}\cdots n_d^{s_d}}\\
&=  \sum_{n\ge n_1\ge \cdots \ge n_w\ge 1}
     (-1)^{n_1}\binom{n}{n_1}
     \dfrac{(1-t_1)^{n_{l(1)}-n_{l(1)+1}} \cdots 
           (1-t_{d-1})^{n_{l(d-1)}-n_{l(d-1)+1}}
            \{ (1-t_d)^{n_{l(d)}} -1 \} }
         {n_1\cdots n_w}.
\end{split}
\end{equation}
\end{thm}

\begin{rem}
 As described in \cite[Remark 2.9]{SS},
 these equations can be derived from 
 Kawashima-Tanaka's formula \cite[Theorem 2.6]{KT}.
\end{rem}

These identities are equivalent to
our identity \eqref{eq:maincor} in Corollary \ref{cor:maincor}, that is,
the following theorem holds.
\begin{thm}
From Eq.~\eqref{eq:maincor},
we can derive Eqs.~\eqref{eq:SS_1} and \eqref{eq:SS_2}, and vice versa.
\end{thm}

\begin{proof}
First we show that 
Eq.~\eqref{eq:maincor} implies
Eqs.~\eqref{eq:SS_1} and \eqref{eq:SS_2}.

For two sequences 
$\{a_n\}_{n\ge 0}$ and $\{b_n\}_{n\ge 0}$, 
the following statement
is well known as the binomial inversion:
\[ b_n=\sum_{k=0}^n\binom{n}{k}a_k
\text{\ \ if and only if\ \ } 
a_n=\sum_{k=0}^n\binom{n}{k}(-1)^{n-k}b_k.\]

By applying $x=-1$ and $y=1$ in Eq.~\eqref{eq:maincor},
we have
\begin{align*}
&\sum_{k=1}^n\binom{n}{k} (-1)^k 
  l_k^{\star}(\boldsymbol{s},\boldsymbol{t})\\
& =\sum_{n\ge n_2\ge \cdots \ge n_w\ge 1}
     \frac{(1-t_1)^{n_{l(1)}-n_{l(1)+1}} \cdots 
     (1-t_{d-1})^{n_{l(d-1)}-n_{l(d-1)+1}} 
     \{   (1-t_d)^{n_{l(d)}}-1 \}}
     {nn_2\cdots n_w}.
\end{align*}
By the binomial inversion, we obtain
\begin{align*}
l_n^{\star}(\boldsymbol{s},\boldsymbol{t})
&= \sum_{n_1=1}^n \binom{n}{n_1}
 \sum_{n_1\ge n_2\ge \cdots \ge n_w\ge 1}
 \frac{(-1)^{n_1}}{n_1n_2\cdots n_w}\\
&\hspace{20pt}\times (1-t_1)^{n_{l(1)}-n_{l(1)+1}} \cdots (1-t_{d-1})^{n_{l(d-1)}-n_{l(d-1)+1}}
  \{  (1-t_d)^{n_{l(d)}} -1\}
\end{align*}
and this proves Eq.~\eqref{eq:SS_2}.

For an index $\boldsymbol{s}=(s_1,\ldots, s_d)$,
define two indices $\boldsymbol{u}$ and $\boldsymbol{v}$ as
\begin{align*}
&\boldsymbol{u}:=
(\overbrace{1,\ldots ,1, 1-t_1}^{s_1},
\ldots, 
\overbrace{1,\ldots ,1, 1-t_{d-1}}^{s_{d-1}},
\overbrace{1,\ldots ,1, 1-t_d}^{s_d}), \\
&\boldsymbol{v}:=
(\overbrace{1,\ldots ,1, 1-t_1}^{s_1},
\ldots, 
\overbrace{1,\ldots ,1, 1-t_{d-1}}^{s_{d-1}},
\overbrace{1,\ldots ,1}^{s_d}).
\end{align*}
Then the right-hand side of \eqref{eq:SS_1}
is 
\[  l_n^{\star}\Big( (\overbrace{1,\ldots,1}^w), \boldsymbol{u} \Big) 
-  l_n^{\star} \Big( (\overbrace{1,\ldots,1}^w), \boldsymbol{v}\Big).
\]
By applying $x=-1$ and $y=1$ in Eq.~\eqref{eq:maincor}
for $l_n^{\star}((1,\ldots, 1), \boldsymbol{u})$
and $l_n^{\star}((1,\ldots, 1), \boldsymbol{v})$,
we have
\begin{align*}
\sum_{k=1}^n\binom{n}{k}(-1)^k
\Big( l_k^{\star}((1,\ldots, 1), \boldsymbol{u})
 - l_k^{\star}((1,\ldots, 1), \boldsymbol{v})\Big)
 =\sum_{n\ge n_2\ge \cdots \ge n_d\ge 1}
\frac{t_1^{n-n_2}  \cdots 
t_{d-1}^{n_{d-1}-n_d} 
t_d^{n_d}
}{n^{s_1} n_2^{s_2} \cdots n_d^{s_d}}.
\end{align*}
By using the binomial inversion again, we obtain Eq.~\eqref{eq:SS_1}.
\vspace{10pt} 

Next we show that
Eqs.~\eqref{eq:SS_1} and \eqref{eq:SS_2} implies 
Eq.~\eqref{eq:maincor}.
By Eq.~\eqref{eq:SS_2}, we have
\begin{align*}
M_n(\boldsymbol{s}, \boldsymbol{t}; x,y)
&=\sum_{k=1}^n \binom{n}{k}x^ky^{n-k}
 \sum_{k\ge n_1\ge \cdots \ge n_w\ge 1}
     (-1)^{n_1}\binom{k}{n_1}\\
&\ \ \  \dfrac{(1-t_1)^{n_{l(1)}-n_{l(1)+1}} \cdots 
           (1-t_{d-1})^{n_{l(d-1)}-n_{l(d-1)+1}}
            \{ (1-t_d)^{n_{l(d)}} -1 \} }
         {n_1\cdots n_w}\\
&=\sum_{n\ge n_1\ge \cdots \ge n_w\ge 1} (-1)^{n_1}
    \sum_{k=n_1}^n \binom{n}{k}\binom{k}{n_1} x^ky^{n-k} \\
&\ \ \  \dfrac{(1-t_1)^{n_{l(1)}-n_{l(1)+1}} \cdots 
           (1-t_{d-1})^{n_{l(d-1)}-n_{l(d-1)+1}}
            \{ (1-t_d)^{n_{l(d)}} -1 \} }
         {n_1\cdots n_w}.
\end{align*}
By direct calculation, we have
\begin{align*}
\sum_{k=n_1}^n \binom{n}{k}\binom{k}{n_1}x^ky^{n-k}
&=\sum_{k=0}^{n-n_1} \binom{n}{n_1}\binom{n-n_1}{k} x^{n_1} x^ky^{n-n_1-k}\\
&= \binom{n}{n_1}x^{n_1}(x+y)^{n-n_1}.    
\end{align*}
Hence we have
\begin{align*}
&M_n(\boldsymbol{s}, \boldsymbol{t}; x,y)\\
&=
\sum_{n\ge n_1\ge \cdots \ge n_w\ge 1}
\binom{n}{n_1}x^{n_1}(x+y)^{n-n_1} (-1)^{n_1} \\
&\ \ \ \dfrac{\left( 1-t_1 \right)^{n_{l(1)}-n_{l(1)+1}}
        \cdots 
       \left( 1-t_{d-1} \right)^{n_{l(d-1)}-n_{l(d-1)+1}}
       \left\{   \left( 1-t_{d} \right)^{n_{l(d)}}-1\right\}}
      {n_1\cdots n_w}\\
&=
\sum_{n\ge n_1\ge \cdots \ge n_w\ge 1}
\binom{n}{n_1}(x+y)^n (-1)^{n_1}
\left(\frac{x}{x+y}\right)^{n_1-n_2} \cdots 
\left(\frac{x}{x+y}\right)^{n_{w-1}-n_w}
\left(\frac{x}{x+y}\right)^{n_w}\\
&\hspace{20pt}\dfrac{\left( 1-t_1 \right)^{n_{l(1)}-n_{l(1)+1}}
        \cdots 
       \left( 1-t_{d-1} \right)^{n_{l(d-1)}-n_{l(d-1)+1}}
       \left\{   \left( 1-t_{d} \right)^{n_{l(d)}}-1\right\}}
      {n_1\cdots n_w}.
\end{align*}
By applying $(s_1,\ldots ,s_d)=(1,\ldots ,1)$ in \eqref{eq:SS_1}, this equals
\begin{align*}
\sum_{n\ge n_1\ge \cdots \ge n_w\ge 1} \dfrac{(x+y)^n}{n_1\cdots n_w} 
\left( \prod_{i=1}^{w-1} Q_i^{n_i-n_{i+1}} \right) 
\left(  \left(1-\frac{x}{x+y}(1-t_d)\right)^{n_{l(d)}}- \left(1-\frac{x}{x+y}\right)^{n_{l(d)}}
\right).
\end{align*}
Here $Q_i$ ($1\le i \le w-1$) are defined as
\begin{align*}
 Q_i:
&=\begin{cases}
    1-\frac{x}{x+y} & \text{if } i\not\in \{l(1),\ldots, l(d-1)\},\\
    1-\frac{x}{x+y}(1-t_j) & \text{if } i=l(j)
    \ \ (1\le j\le d-1).
\end{cases}\\
&=\frac{1}{x+y} \times \begin{cases}
    y& \text{if } i\not\in \{l(1),\ldots, l(d-1)\},\\
    (t_jx+y)& \text{if } i=l(j)\ \ (1\le j\le d-1).
\end{cases}
\end{align*}
Therefore we obtain that
\begin{align*}
&M_n(\boldsymbol{s}, \boldsymbol{t}; x,y)\\
&=\sum_{n\ge n_1\ge \cdots \ge n_w\ge 1}
\dfrac{(x+y)^{n-n_1} y^{n_1+n_{l(1)+1}+\cdots n_{l(d-1)+1}
 - n_{l(1)}-\cdots -n_{l(d)}}}
      {n_1\cdots n_w}  \\
&\times \left( t_1 x+y \right)^{n_{l(1)}-n_{l(1)+1}}
           \cdots
         \left( t_{d-1}x+y\right)^{n_{l(d-1)}-n_{l(d-1)+1}}
            \big( (t_dx +y)^{n_{l(d)}} -y^{n_{l(d)}}\big)
\end{align*}
and this proves Eq.~\eqref{eq:maincor}.

\end{proof}

\section{The Cauchy binomial sums}\label{sec:Cauchybt}

For a positive integer $n$ and complex constants $\alpha$ and $\beta$, define
\[ (\alpha x+\beta)^{[n]}:= \prod_{k=1}^n (q^{k-1}\alpha x+\beta)
=(\alpha x+\beta)(q\alpha x+\beta) \cdots 
(q^{n-1}\alpha x+\beta)\ \ (n\ge 0).\]
When $q$ tends to $1$, the function
$(\alpha x+\beta)^{[n]}$ tends to the ordinary power $(\alpha x+\beta)^n$.
Under these notations, the Cauchy binomial theorem is stated as follows:
\begin{align}\label{eq:Cauchybinomialthm}
(x+a)^{[n]} = \sum_{k=0}^n \qbinom{n}{k} q^{\binom{k}{2}}
x^k a^{n-k}\ \ (n\ge 0),
\end{align}
where $a$ is a fixed complex number.

It is known that the Cauchy binomial theorem \eqref{eq:Cauchybinomialthm}
and Eq.~\eqref{eq:q-binomial_nc} are equivalent.
In fact, for a variable $a$ which commutes with $x$ and $y$,
we make the substitution 
$x\mapsto xy$ and $y\mapsto ay$ in Eq.~\eqref{eq:q-binomial_nc}.
This substitution is allowed
because $(ay)(xy)= q(xy)(ay)$ holds
and it leads to \eqref{eq:Cauchybinomialthm}.
Conversely, Eq.~\eqref{eq:q-binomial_nc}
can also be obtained from \eqref{eq:Cauchybinomialthm}.
In detail, see \cite[Sec.~1]{J} and \cite[Sec.~2]{K}.

The following lemma, which will be used later, can be proved by direct calculation.
\begin{lemma}\label{lemma:(xy)^k}
For integers $k$, $m$, $n\ge 0$, 
the following identities hold in $\mathbb{C}_q[x,y]$:
\begin{enumerate}
  \item $(xy)^k =q^{\binom{k}{2}}x^ky^k$,
  \item $y^m((tx+a)y)^n= (q^mtx+a)^{[n]}y^{m+n}$.
\end{enumerate}

\end{lemma}

The following is the Boyadzhiev-Mneimneh-type theorem for 
$l_n^{\star, q}$ corresponding to the Cauchy binomial theorem.

\begin{thm}\label{thm:M_sum_Cauchy}
For $\boldsymbol{s}=(s_1,\ldots ,s_d)\in \mathbb{Z}_{>0}^d$
and an integer $n\ge 1$, it holds that
\begin{align*}
&    \sum_{k=1}^n \qbinom{n}{k} q^{\binom{k}{2}} x^ka^{n-k}
l_k^{\star, q}(\boldsymbol{s}, \boldsymbol{t})\\
&= \sum_{n\ge n_1\ge \cdots \ge n_w\ge 1}
\frac{(x+a)^{[n-n_1]}}{ [n_1]\cdots [n_w]}
 \left(\prod_{r=1}^{d-1} 
 (q^{n-n_{l(r)}}t_rx+a)^{[n_{l(r)}-n_{l(r)+1}]}\right) \\
&\big( (q^{n-n_{l(d)}}t_dx+a)^{[n_{l(d)}]} -a^{n_{l(d)}}\big)
a^{n_1+n_{l(1)+1}+\cdots n_{l(d-1)+1}-n_{l(1)}-\cdots
- n_{l(d)}}
\end{align*}

\end{thm}

\begin{proof}
Following the above argument, 
we make the substitution $x\mapsto xy$ and $y\mapsto ay$
in Theorem \ref{thm:mainthm}.
Then, by using Lemma \ref{lemma:(xy)^k} (i),
the left-hand side of \eqref{eq:mainthm}
becomes 
\[    \sum_{k=1}^n \qbinom{n}{k} q^{\binom{k}{2}} x^ky^n a^{n-k}
l_k^{\star, q}(\boldsymbol{s}, \boldsymbol{t}).
\]
The right-hand side of \eqref{eq:mainthm} becomes 
\begin{align*}
&\sum_{n\ge n_1\ge \cdots \ge n_w\ge 1}
\frac{((x+a)y)^{n-n_1}}{ [n_1]\cdots [n_w]}
\left( \prod_{r=1}^{d-1}  
y^{n_{l(r-1)+1}-n_{l(r)}} ((t_rx+a)y)^{n_{l(r)}-n_{l(r)+1}}
\right)\\
&\times y^{n_{l(d-1)+1}-n_{l(d)}}
\left(    ((t_dx+a)y)^{n_{l(d)}} - (ay)^{n_{l(d)}}\right)
 a^{n_1+n_{l(1)+1}+\cdots n_{l(d-1)+1}-n_{l(1)}-\cdots
- n_{l(d)}}.
\end{align*}
By using Lemma \ref{lemma:(xy)^k} (ii) and 
moving all $y$'s to the right, a part of the above equation can be written as follows:
\begin{align*}
&((x+a)y)^{n-n_1}
\left( \prod_{r=1}^{d-1}  
y^{n_{l(r-1)+1}-n_{l(r)}} ((t_rx+a)y)^{n_{l(r)}-n_{l(r)+1}}
\right)\\
&\times y^{n_{l(d-1)+1}-n_{l(d)}}
\left(    ((t_dx+a)y)^{n_{l(d)}} - (ay)^{n_{l(d)}}\right)\\
&=
(x+a)^{[n-n_1]} \left(\prod_{r=1}^{d-1} 
 (q^{n-n_{l(r)}}t_rx+a)^{[n_{l(r)}-n_{l(r)+1}]}\right)
\left(  (q^{n-n_{l(d)}}t_dx+a)^{[n_{l(d)}]} - a^{n_{l(d)}}\right) y^n.
\end{align*}
Therefore, we obtain that
\begin{align*}
&    \sum_{k=1}^n \qbinom{n}{k} q^{\binom{k}{2}} x^ky^n a^{n-k}
l_k^{\star, q}(\boldsymbol{s}, \boldsymbol{t})\\
&= \sum_{n\ge n_1\ge \cdots \ge n_w\ge 1}
\frac{(x+a)^{[n-n_1]}}{ [n_1]\cdots [n_w]}
 \left(\prod_{r=1}^{d-1} 
 (q^{n-n_{l(r)}}t_rx+a)^{[n_{l(r)}-n_{l(r)+1}]}\right)\\
&\left(  (q^{n-n_{l(d)}}t_dx+a)^{[n_{l(d)}]} - a^{n_{l(d)}}\right)
a^{n_1+n_{l(1)+1}+\cdots n_{l(d-1)+1}-n_{l(1)}-\cdots - n_{l(d)}}
y^n.
\end{align*}
By canceling $y^n$ from both sides,
we obtain the desired equation.
\end{proof}

\begin{rem}
Other types of formulas for $q$-analogues of multiple harmonic sums, such as the duality relation
\cite[Theorem 1]{B} and its extension,
the Ohno type identity \cite[Theorem 2.1]{SY}, are known.
These formulas are generalizations of
\cite[Theorem 4]{D} and \cite[Theorem~2]{P}.
\end{rem}
\vspace{5pt}

\begin{rem}
Bradley \cite[Corollary~3]{B} proved an
identity
\[ \sum_{k=1}^n(-1)^{k+1} q^{\binom{k}{2}}
\qbinom{n}{k} \sum_{k\ge k_1\ge \cdots \ge k_d\ge 1}
\frac{1}{[k_1]\cdots [k_d]}
= \frac{1}{[n]^d}\ \ (n, d \ge 1).\]
This identity can be also obtained from our 
Theorem \ref{thm:M_sum_Cauchy}
by applying $x=-1$, $a=1$, $\boldsymbol{s}=(1,\ldots ,1)$
and $\boldsymbol{t}=(1,\ldots ,1)$.
\end{rem}
\vspace{10pt}

\noindent \textbf{\large Acknowledgment}\\[5pt]
This work was supported by JSPS KAKENHI Grant Number 25K06944.

\vspace{10pt}

\noindent \textbf{Address}: Department of Robotics, Osaka Institute of Technology\\
1-45 Chaya-machi, Kita-ku, Osaka 530-8585, Japan\\

\noindent \textbf{E-mail}: ken.kamano@oit.ac.jp


\vspace{10pt}




\begin{thebibliography}{10}

\bibitem{Boya}
K.~N.~Boyadzhiev:
Harmonic number identities via Euler's transform,
J.~Integer Seq. \textbf{12} (2009), Article 09.6.1.

\bibitem{B}
D.~M.~Bradley:
Duality for finite multiple harmonic $q$-series,
Discrete Math. \textbf{300} (2005), 44--56.


\bibitem{C}
J.~M.~Campbell:
On Mneimneh's binomial sum involving harmonic numbers, Discrete Math. \textbf{346} (2023), 113549.

\bibitem{D}
K.~Dilcher:
Some $q$-series identities related to divisor functions,
Discrete Math. \textbf{145} (1995), 83--93.

\bibitem{G}
M.~Gen\v{c}ev:
Pan-Xu conjecture and reduction formulas for polylogarithms,
arXiv: 2408.16148.

\bibitem{J}
W.~P.~Johnson:
$q$-extensions of identities of Abel-Rothe type,
Discrete Math. \textbf{159} (1996), 161--177.


\bibitem{K}
T.~H.~Koornwinder:
Special functions and $q$-commuting variables, 
arXiv: 9608008.


\bibitem{KT}
G.~Kawashima and T.~Tanaka:
Newton series and extended derivation relations for multiple
$L$-values, arXiv:0801.3062.


\bibitem{KW}
T.~Komatsu and P.~Wang:
A generalization of Mneimneh's binomial sum of harmonic
numbers,
Discrete Math. \textbf{347} (2024), 113945.

\bibitem{M}
S.~Mneimneh: A binomial sum of harmonic numbers, Discrete Math. \textbf{346} (2023), 113075.
 
\bibitem{PX}
E.~Pan and Ce Xu:
Mneimneh-type binomial sums of multiple harmonic-type sums,
arXiv:2403.17952.

\bibitem{P}
H.~Prodinger:
A $q$-analogue of a formula of Hernandez obtained by inverting a result of Dilcher,
Australas. J. Combin. \textbf{21} (2000), 271--274.


\bibitem{SS}
K.~Sakugawa and S.~Seki: 
On functional equations of finite multiple polylogarithms,
J. Algebra \textbf{469} (2017), 323--357.


\bibitem{S}
M.P.~Sch\"{u}tzenberger: Une interpr\'{e}tation de certaines solutions de l'\'{e}quation fonctionnelle: $F(x +
y) = F(x)F(y)$,
C.~R.~Acad.~Sci.~Paris \textbf{236} (1953), 352--353.

\bibitem{SY}
S.~Seki and S.~Yamamoto:
Ohno-type identities for multiple harmonic sums,
J.~Math.~Soc.~Japan \textbf{72} (2020),673--686.

\bibitem{TZ}
R.~Tauraso and J.~Zhao:
Congruences of alternating multiple harmonic sums,
J.~Comb.~Number Theory \textbf{2} (2010), 129--159.

\end{thebibliography}
\end{document}